\documentclass[twoside]{amsart}
\usepackage{latexsym}
\usepackage{amssymb,amsmath,amsopn}
\usepackage[dvips]{graphicx}   
\usepackage{color,epsfig}      
\usepackage{tikz} 
\usepackage{subcaption} 
\usepackage{eurosym}
\newtheorem{thm}{Theorem}

\newtheorem{lem}{Lemma}
\newtheorem{cor}{Corollary}

\theoremstyle{definition}
\newtheorem{defn}{Definition}
\newtheorem{rem}{Remark}
\newtheorem{ques}{Question}
\newtheorem{conj}{Conjecture}
\newtheorem{prob}{Problem}

\renewcommand{\Re}{\mathbb R}

\newcommand{\B}{\mathbf B}

\newcommand{\Sph}{\mathbb{S}}

\newcommand{\E}{\mathcal{E}}
\renewcommand{\O}{\mathcal{O}}

\DeclareMathOperator{\inter}{int}
\DeclareMathOperator{\bd}{bd}

\DeclareMathOperator{\conv}{conv}

\DeclareMathOperator{\proj}{proj}
\DeclareMathOperator{\perim}{perim}

\DeclareMathOperator{\diam}{diam}

\DeclareMathOperator{\relint}{relint}
\DeclareMathOperator{\vol}{vol}

\begin{document}

\parskip=4pt

\title[Monostable polyhedra]{A solution to some problems of Conway and Guy on monostable polyhedra}

\author[Z. L\'angi]{Zsolt L\'angi}

\thanks{The author is supported by the National Research, Development and Innovation Office, NKFI, K-134199, the J\'anos Bolyai Research Scholarship of the Hungarian Academy of Sciences, the \'UNKP-20-5 New National Excellence Program by the Ministry of Innovation and Technology and the NRDI Fund (TKP2020 IES, Grant No. TKP2020 BME-IKA-VIZ).}

\address{MTA-BME Morphodynamics Research Group and Department of Geometry, Budapest University of Technology, Egry J\'ozsef utca 1., Budapest 1111, Hungary}  \email{zlangi@math.bme.hu}



\begin{abstract}
A convex polyhedron is called monostable if it can rest in stable position only on one of its faces. The aim of this paper is to investigate three questions of Conway, regarding monostable polyhedra, which first appeared i a 1969 paper of Goldberg and Guy (M. Goldberg and R.K. Guy, \emph{Stability of polyhedra (J.H. Conway and R.K. Guy)}, SIAM Rev. \textbf{11} (1969), 78-82). In this note we answer two of these problems and make a conjecture about the third one. The main tool of our proof is a general theorem describing approximations of smooth convex bodies by convex polyhedra in terms of their static equilibrium points. As another application of this theorem, we prove the existence of a convex polyhedron with only one stable and one unstable point.
\end{abstract}

\maketitle


\section{Introduction}\label{sec:intro}

The study of static equilibrium points of convex bodies started with the work of Archimedes \cite{archimedes}, and has been continued throughout the history of science in various disciplines: from geophysics and geology \cite{Pebbles,ceug} leading to examination of the possible existence of water on Mars \cite{Mars}, to robotics and manufacturing \cite{Varkonyi,Bohringer} to biology and medicine \cite{pill,dpill,turtles}.
In modern times, the mathematical aspects of this concept was started by a problem of Conway and Guy \cite{conway} in 1966 who conjectured that there is no homogeneous tetrahedron which can stand in rest only on one of its faces when placed on a horizontal plane, but there is a homogeneous convex polyhedron with the same property. These two questions were answered by Goldberg and Guy in \cite{goldberg} in 1969, respectively (for a more detailed proof of the first problem, see \cite{Dawson}), who called the convex polyhedra satisfying this property \emph{monostable} or \emph{unistable}. In addition, in \cite{goldberg} Guy presented some problems regarding monostable polyhedra, stating that three of them are due to Conway (for similar statements in the literature, see e.g. \cite{CFG,DawsonFinbow,dumitrescu}). These three questions appear also in the problem collection of Croft, Falconer and Guy \cite{CFG} as Problem B12. The aim of our paper is to examine these problems.

\begin{prob}\label{prob1}
Can a monostable polyhedron in the Euclidean $3$-space $\Re^3$ have an $n$-fold axis of symmetry for $n > 2$?
\end{prob}

Before the next problem, recall that the \emph{girth} of a convex body in $\Re^3$ is the minimum perimeter of an orthogonal projection of the body onto a plane \cite{dumitrescu}.

\begin{prob}\label{prob2}
What is the smallest possible ratio of diameter to girth for a monostable polyhedron?
\end{prob}

\begin{prob}\label{prob3}
What is the set of convex bodies uniformly approximable by monostable polyhedra, and does this contain the sphere?
\end{prob}

It is worth noting that, according to Guy \cite{goldberg}, Conway showed that no body of revolution can be monostable, and also that the polyhedron constructed in \cite{goldberg} has a $2$-fold rotational symmetry. Problems~\ref{prob1}-\ref{prob3} appear also in the problem collection of Klamkin \cite{Klamkin}, and Problem~\ref{prob1} and  some other problems for monostable polyhedra appear in a 1968 collection of geometry problems of Shephard \cite{Shephard}, who described these objects as `a remarkable class of convex polyhedra' whose properties `it would probably be very rewarding and interesting to make a study of'.

Before stating our main result, we collect some related results and problems from the literature. Here, we should first mention from \cite{goldberg} the problem of finding the minimal dimension $d$ in which a $d$-simplex can be monostable. This problem have been investigated by Dawson et al. \cite{Dawson,Dawson2,Dawson3,DawsonFinbow}, who proved that there is no monostable $d$-simplex if $d \leq 8$ and there is a monostable $11$-simplex. With regard to Problem XVI in \cite{Shephard}, asking about the minimum number of faces of a monostable polyhedron in $\Re^3$, the original construction of Guy \cite{goldberg} with $19$ faces (attributed also to Conway) was modified by Bezdek \cite{Bezdek} to obtain a monostable polyhedron with $18$ faces, while a computer-aided search by Reshetov \cite{Reshetov} yields a monostable polyhedron with $14$ faces. In \cite{Heppes}, Heppes constructed a homogeneous tetrahedron in $\Re^3$ with the property that putting it on a horizontal plane with a suitable face, it rolls twice before finding a stable position. Another interesting convex body is found by Dumitrescu and T\'oth \cite{dumitrescu}, who constructed a convex polyhedron $P$ with the property that after placing it on a horizontal plane with a suitable face, it covers an arbitrarily large distance while rolling until it finds a stable position. Finally, we remark that a systematic study of the equilibrium properties of convex polyhedra was started in \cite{DKLRV}.

Our main result is the following, where $d_H(\cdot, \cdot)$ and $\B^3$ denotes the Hausdorff distance of convex bodies, and the closed unit ball in $\Re^3$ centered at the origin $o$, respectively.

\begin{thm}\label{thm:rotational}
For any $n \geq 3$, $n \in \mathbb{Z}$ and $\varepsilon > 0$ there is a homogeneous monostable polyhedron $P$ such that $P$ has an $n$-fold rotational symmetry and
$d_H(P, \B^3) < \varepsilon$.
\end{thm}

Theorem~\ref{thm:rotational} answers Problem~\ref{prob1}, and also the case of a sphere in Problem~\ref{prob3}. In addition, from it we may deduce Corollary~\ref{cor:girth} (for a lower bound without proof, see also \cite{dumitrescu}). This solves Problem~\ref{prob2}.
Here, for any convex body $K \subset \Re^3$, we denote by $\diam(K)$ and $g(K)$ the \emph{diameter} and the \emph{girth} of $K$, respectively.

\begin{cor}\label{cor:girth}
The infimum of the quantities $\frac{\diam(P)}{g(P)}$ over the family of all monostable polyhedra $P$ is $\frac{1}{\pi}$.
\end{cor}

The proof of Theorem~\ref{thm:rotational} is based on a general theorem on approximation of smooth convex bodies by convex polytopes. 
Before stating it, we briefly introduce some elementary concepts regarding their equilibrium properties. Let $K \subset \Re^3$ be a smooth convex body with the origin as its center of mass, and let $\delta_K : \bd(K) \to \Re$ be the Euclidean distance function measured from $o$, where $\bd(K)$ denotes the boundary of $K$. The critical points of $\delta_K$ are called \emph{equilibrium points} of $K$. To avoid degeneracy, it is usually assumed that $\delta_K$ is a Morse function; i.e. it has finitely many critical points, $\bd(K)$ is twice continuously differentiable at least in a neighborhood of each critical point, and at each such point the Hessian of $\delta_K$ is nondegenerate \cite{Zomorodian}. Depending on the number of negative eigenvalues of the Hessian, we distinguish between \emph{stable}, \emph{unstable} and \emph{saddle-type} equilibrium points, corresponding to the local minima, maxima and saddle points of $\delta_K$, respectively. The Poincar\'e-Hopf Theorem implies that under these conditions, the numbers $S$, $U$ and $H$ of the stable, unstable and saddle points of $K$, respectively, satisfy the equation $S-H+U=2$.

Answering a conjecture of Arnold, Domokos and V\'arkonyi \cite{gomboc} proved that there is a homogeneous convex body with only one stable and one unstable point. They called the body they constructed `G\"omb\"oc' (for more information, see \cite{gomboc_homepage}).
In addition to the existence of G\"omb\"oc, in their paper \cite{gomboc} Domokos and V\'arkonyi proved the existence of a convex body with $S$ stable and $U$ unstable equilibrium points for any $S,U \geq 1$. This investigation was extended in \cite{DLS2} to the combinatorial equivalence classes defined by the Morse-Smale complexes of $\rho_K$, and in \cite{DHL} for transitions between these classes. Based on these results, for any $S, U \geq 1$ we define the set $(S,U)_c$ as the family of smooth convex bodies $K$ having $S$ stable and $U$ unstable equilibrium points, where $K$ has no degenerate equilibrium point, and at each such point $\bd(K)$ has a positive Gaussian curvature. We define the class $(S,U)_p$ analogously for convex polyhedra, where stable and unstable points of a convex polyhedron are defined formally in Section~\ref{sec:prelim}. Our theorem is the following, where we call a convex body \emph{centered} if its center of mass is the origin $o$.

\begin{thm}\label{thm:approx}
Let $\varepsilon > 0$, $S,U \geq 1$ be arbitrary, and let $G$ be any subgroup of the orthogonal group $\O(3)$. Then for any centered, $G$-invariant convex body $K \in (S,U)_c$, there is a centered $G$-invariant convex polyhedron $P \in (S,U)_p$ such that $d_H(K,P) < \varepsilon$.
\end{thm}

Here we note that the fact that any nondegenerate convex polyhedron can be approximated arbitrarily well by a smooth convex body with the same number of equilibrium points is regarded as `folklore' (we use a simple argument to show it in Section~\ref{sec:prelim}). On the other hand, it is shown in \cite{DLS} that any sufficiently fine approximation of a smooth convex body $K$ by a convex polyhedron $P$, using an equidistant partition of the parameter range of the boundary of $K$, has strictly more stable, unstable and saddle points in general than the corresponding quantities for $K$.

Even though the convex body constructed in \cite{gomboc} is not $C^2$-class at its two equilibrium points, in \cite{DLS2} it is shown that class $(1,1)_c$ is not empty.
Thus, Theorem~\ref{thm:approx} readily implies the existence of a polyhedron in class $(1,1)_p$.

\begin{cor}\label{cor:gomboc}
There is a convex polyhedron with a unique stable and a unique unstable point.
\end{cor}

Furthermore, we remark that the elegant construction in the paper \cite{dumitrescu} of Dumitrescu and T\'oth yields an \emph{inhomogeneous} monostable convex polyhedron arbitrarily close to a sphere. Nevertheless, we must add that dropping the requirement of uniform density may significantly change the equilibrium properties of a convex body. To show it we recall the construction of Conway (see \cite{DawsonFinbow}) of an inhomogeneous monostable tetrahedron in $\Re^3$, and observe that spheres with inhomogeneous density, as also roly-poly toys, yield trivial solutions to Arnold's conjecture.

For completeness, we also recall the remarkable result of Zamfirescu \cite{Zamfirescu} stating that a typical convex body (in Baire category sense) has infinitely many equilibrium points, and note that critical points of the distance function from another point are examined in Riemannian manifolds, e.g. in \cite{BaranyZamfirescu,Garcia,IVZ}.

In Section~\ref{sec:prelim}, we introduce our notation and collect the necessary tools for proving Theorems~\ref{thm:rotational} and \ref{thm:approx}, including more precise definitions for some of the concepts mentioned in Section~\ref{sec:intro}. In Section~\ref{sec:proof1} we prove Theorem~\ref{thm:rotational} and Corollary~\ref{cor:girth}. In Section~\ref{sec:proof2} we prove Theorem~\ref{thm:approx}. Finally, in Section~\ref{sec:remarks} we collect some additional remarks and ask some open questions.

\section{Preliminaries}\label{sec:prelim}

In the paper, for any $p,q \in \Re^3$, we denote by $[p,q]$ the closed segment with endpoints $p,q$, and by $|p|$ the Euclidean norm of $p$. We denote the closed $3$-dimensional unit ball centered at the origin $o$ by $\B^3$, and its boundary by $\Sph^2$. Furthermore, for any set $S \subset \Re^3$ we let $\conv (S)$ denote the convex hull of $S$. By a convex body we mean a compact, convex set with nonempty interior.

Let $K \subset \Re^3$ be a convex body. The \emph{center of mass} $c(K)$ of $K$ is defined by the fraction $c(K)=  \frac{1}{\vol(K)} \int_{x \in K} x \, d v$, where $v$ denotes $3$-dimensional Lebesgue measure. We remark that the integral in this definition is called the \emph{first moment} of $K$, and note that we clearly have $c(K) \in \inter (K)$ for any convex body $K$. If $q \in \bd(K)$ satisfies the property that the plane through $q$ and orthogonal to the vector $q-c(K)$ supports $K$, then we say that $q$ is an \emph{equilibrium point} of $K$. Here, if $K$ is smooth, then the equilibrium points of $K$ coincide with the critical points of the Euclidean distance function measured from $c(K)$ and restricted to $\bd (K)$.
We remark that a convex body $K \subset \Re^3$ is called \emph{smooth} if for any boundary point $x$ of $K$ there is a unique supporting plane of $K$ at $q$; this property coincides with the property that $\bd(K)$ is  a $C^1$-class submanifold of $\Re^3$ (cf. \cite{Schneider}).

We define nondegenerate equilibrium points only in two special cases. If $K$ is smooth, $q \in \bd(K)$ is an equilibrium point of $K$ with a $C^2$-class neighborhood in $\bd(K)$, and the Hessian of the Euclidean distance function on $\bd(K)$, measured from $c(K)$, is nondegenerate, we say that $q$ is \emph{nondegenerate}. In this case $q$ is called a \emph{stable, saddle-type} or \emph{unstable} point of $K$ if the number of the negative eigenvalues of the Hessian at $q$ is $0,1$ or $2$, respectively \cite{DLS2}. Consider now the case that $K$ is a convex polyhedron in $\Re^3$, and $q \in \bd(K)$ is an equilibrium point of $K$. Then there is a unique vertex, edge or face of $K$ that contains $q$ in its relative interior, where by the relative interior of a vertex we mean the vertex itself. Let $F$ denote this vertex, edge or face, and let $H$ be the supporting plane of $K$ through $q$ that is perpendicular to $q-c(K)$. Observe that $F \subset K \cap H$. We say that $q$ is \emph{nondegenerate} if $F= K \cap H$. In this case we call $q$ a \emph{stable, saddle-type} or \emph{unstable} point of $K$ if the dimension of $F$ is $2,1$ or $0$, respectively \cite{DKLRV}. In both the smooth and the polyhedral cases $K$ is called nondegenerate if it has only finitely many equilibrium points, and each such point is nondegenerate; note that the first condition is automatically satisfied for convex polyhedra. We remark that in the above definitions, we may replace the center of mass of $K$ by any fixed reference point $c \in K$. In this case we write about equilibrium points \emph{relative to $c$}. We emphasize that in the paper, unless it is stated otherwise, if the reference point is not specified, then it is meant to be the center of mass of the body.

Let $K \in \Re^3$ be a nondegenerate smooth convex body with $S$ stable, $H$ saddle-type and $U$ unstable equilibrium points. Using a standard convolution technique, we may assume that $K$ has a $C^{\infty}$-class boundary, and hence, by the Poincar\'e-Hopf Theorem, we have $S-H+U=2$ \cite{DLS2}. We show that the same holds if $K$ is a nondegenerate convex polyhedron. Indeed, let $\tau > 0$ be sufficiently small, and set $K(\tau) = (K \div (\tau \B^3)) + (\tau \B^3)$, where $\div$ denotes Minkowski difference and $+$ denotes Minkowski addition \cite{Schneider}. Then for any $\tau > 0$, $K(\tau)$ is a smooth nondegenerate convex body having the same numbers of stable, saddle-type and unstable points relative to $c(K)$; hence, we may apply the Poincar\'e-Hopf Theorem for $K(\tau)$ (here we note that by Lemma~\ref{lem:approximation} the same property holds relative to $c(K(\tau))$ as well). Thus, for any nondegenerate convex body, the numbers of stable and unstable points determine the number of saddle-type points. We define class $(S,U)_c$ as the family of nondegenerate, smooth convex bodies $K \subseteq \Re^3$ with $S$ stable, $U$ unstable points with the additional assumption that at each equilibrium point of $K$, the principal curvatures of $\bd(K)$ are positive. Similarly, by $(S,U)_p$ we mean the family of nondegenerate convex polyhedra with $S$ stable and $U$ unstable points. Observe that if $K$ is nondegenerate, the point of $\bd(K)$ closest to or farthest from $c(K)$ is necessarily a stable or unstable point, respectively, implying that the numbers $S,U$ in the above symbol are necessarily positive.

For the following remark, see Lemma 7 from \cite{DLS2}.

\begin{rem}\label{rem:stability}
Let $K \in (S,U)_c$ and for any equilibrium point $q$ of $K$, let $V_q$ be an arbitrary compact neighborhood of $q$ containing no other equilibrium point of $K$. Then $c(K)$ has an open neighborhood $U$ such that for any $x \in U$, $K$ has $S$ stable and $U$ unstable points relative to $x$, and for any equilibrium point $q$ of $K$ relative to $c(K)$, $V_q$ contains exactly one equilibrium point of $K$ relative to $x$, and the type of this point is the same as the type of $q$.
\end{rem}

For Remark~\ref{rem:curvature}, see the paragraph in \cite{DLS} after Definition 2.

\begin{rem}\label{rem:curvature}
Let $q$ be an equilibrium point of a centered convex body $K$ in $(S,U)_c$ for some $S,U \geq 1$. Let $| q | = \rho$, and let $\kappa_1, \kappa_2$ denote the principal curvatures of $\bd (K)$ at $q$. Then $\kappa_1, \kappa_2 \neq \frac{1}{\rho}$. Furthermore, $0 \leq \kappa_1, \kappa_2 < \frac{1}{\rho}$ if and only if $q$ is a stable point, $\kappa_1, \kappa_2 > \frac{1}{\rho}$ if and only if $q$ is an unstable point, and $0 < \min \{ \kappa_1, \kappa_2 \} < \frac{1}{\rho} < \max \{ \kappa_1, \kappa_2 \}$ if and only if $q$ is a saddle-type equilibrium point.
\end{rem}

\begin{lem}\label{lem:discrete}
The symmetry group of any nondegenerate convex body $K$ is finite.
\end{lem}

\begin{proof}
Let $K$ be a nondegenerate convex body with symmetry group $G$. Without loss of generality, assume that $K$ is centered, i.e. $c(K)=o$.
Since $c(K)$ is clearly a fixed point of any symmetry in $G$, we have that $G$ is a subgroup of the orthogonal group $\O(3)$.
Clearly, $G$ is closed in $\O(3)$, and thus, it is a Lie group embedded in $\O(3)$ by Cartan's Closed Subgroup Theorem. On the other hand, the Lie subgroups of $\O(3)$ are well known, and in particular we have that if $G$ is infinite, then it contains, up to conjugacy, $\mathcal{SO}(2)$ as a subgroup. In other words, $K$ is rotationally symmetric. Thus, by nondegeneracy, $K$ has exactly one stable and one unstable equilibrium point. But this property contradicts Conway's result mentioned in Section~\ref{sec:intro} that no rotationally symmetric convex body is monostable.
\end{proof}

We finish Section~\ref{sec:prelim} with two lemmas and two remarks, where $X \triangle Y$ denotes the symmetric difference of the sets $X, Y$.

\begin{lem}\label{lem:approximation}
Let $K(\tau) \subset \Re^3$ be a $1$-parameter family of convex bodies, where $\tau \in [0,\tau_0]$ for some $\tau_0 > 0$. For any $\tau \in [0,\tau_0]$, let $c(\tau)$ denote the center of mass of $K(\tau)$, and let $K=K(0)$ and $c=c(0)$. Assume that for some $C > 0$ and $m > 0$, $\vol(K(\tau) \triangle K) \leq C \tau^m$ holds for any sufficiently small value of $\tau$. Then there is some $C' > 0$ such that $|c(\tau)-c| \leq C' \tau^m$ holds for any sufficiently small value of $\tau$.
\end{lem}

\begin{proof}
Without loss of generality, we may assume that $K(\tau) \subseteq r \B^3$ for some suitable value of $r > 0$ if $\tau$ is sufficiently small.
By definition, $c(\tau) =  \frac{\int_{x \in K(\tau)} x \, d v}{\vol(K(\tau))}$. On the other hand, by the conditions, we have $|\vol(K(\tau))-\vol(K)| \leq C \tau^m$, and $|\int_{x \in K(\tau)} x \, d \lambda - \int_{x \in K} x \, d v | \leq r C \tau^m$ for all sufficiently small values of $\tau$. From these inequalities and the fact that $\vol(K) > 0$, the assertion readily follows.
\end{proof}

\begin{lem}\label{lem:conway_lemma}
Let $p \in \inter (\B^2) \subset \Re^2$ and $q \in \Sph^1$ such that $p$, $q$ and $o$ are not collinear, and let $L$ be a line through $p$ such that $L$ does not separate $o$ and $q$.
Furthermore, if $A$ denotes the convex angular region with $q \in A$ and bounded by a half line of $L$ starting at $p$, and the half line starting at $p$ and containing $o$, then assume that the angle of $A$ is obtuse. Then there is a convex polygon $Q \subset \B^2$ with vertices $o, x_0=q, x_1, \ldots, x_k=p$ in cyclic order in $\bd (Q)$ such that $x_{i-1}x_io \angle > \frac{\pi}{2}$ for all values of $i$, and $L$ supports $Q$.
\end{lem}

\begin{figure}[ht]
\begin{center}
\includegraphics[width=0.35\textwidth]{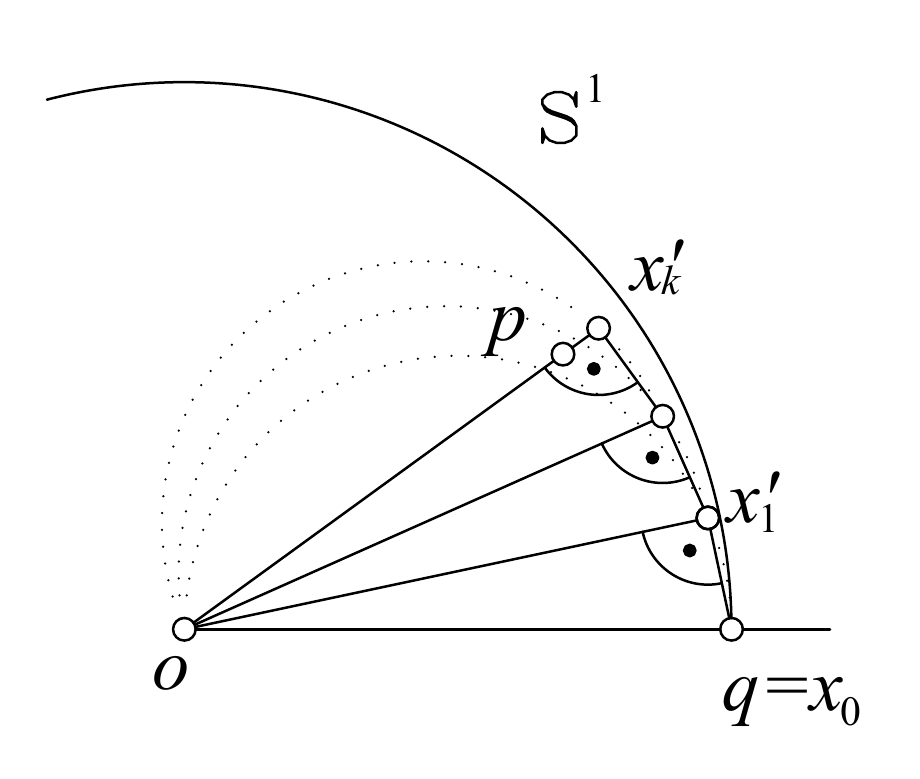}
\caption{The construction of the points $x_i'$ in the proof of Lemma~\ref{lem:conway_lemma}. The dotted curves indicate arcs in the Thales circles of the segments $[o,x_i']$.}
\label{fig:conway_lemma}
\end{center}
\end{figure}

We remark that the conditions in Lemma~\ref{lem:conway_lemma} imply that the Euclidean distance function $x \mapsto |x|$, $x \in \Re^2$ strictly decreases along the curve $\bigcup_{i=1}^k [x_{i-1},x_i]$ from $q$ to $p$. 

\begin{proof}
Without loss of generality, we may assume that $q=(1,0)$ and the $y$-coordinate of $p$ is positive. Set $poq \angle = \beta \in (0,\pi)$, and choose an arbitrary positive integer $k$. For any $i=0,1,\ldots,k$, define the point $x_i' = \left( r_i \cos \frac{i \beta}{k}, r_i \sin \frac{i \beta}{k} \right)$, where $r_i = \cos^{i} \frac{\beta}{k}$. Then $x'_0=q$, and $x'_i$ is on the Thales circle of the segment $[0,x'_{i-1}]$, and thus, $x_{i-1}'x_i'o \angle = \frac{\pi}{2}$ (cf. Figure~\ref{fig:conway_lemma}) for all $i=1,2,\ldots, k$. Using elementary calculus, we obtain that $\lim_{k \to \infty} \cos^k \frac{\beta}{k} = 1$, which yields that there is some value of $k$ such that $|x'_k| > |p|$. Since $x'_k$ and $p$ are on the same half line, we may decrease the values of $r_i$ for $i=1,2,\ldots, k$ slightly such that for the points $x_i$ obtained in this way the convex polygon $Q=\conv \{ o, x_0, x_1, \ldots, x_k \}$ satisfies the required conditions apart from the one for $L$.
Now, if $L$ supports $Q$, we are done. On the other hand, if $L$ does not support $Q$, then we may take the polygon obtained as the intersection of $Q$ and the closed half plane bounded by $L$ and containing $o$ in its interior.
\end{proof}

\begin{rem}\label{rem:forsaddle}
Let $a, b > 0$, where $a \neq b$, and let $E \subset \Re^2$ be the ellipse with equation $\frac{x^2}{a^2} + \frac{y^2}{b^2} \leq 1$. Then, for any $\delta > 0$ there is some $\varepsilon > 0$ such that if $K \subset \Re^2$ is a plane convex body satisfying $E \subseteq K \subseteq (1+\delta)E$, and the vector $w$ is perpendicular to a supporting line of $K$ through $w \in \bd (K)$, then the angle between $w$ and the $x$-axis or the $y$-axis is at most $\delta$.
\end{rem}

\begin{rem}\label{rem:monotone}
Let $f,g$ be two real functions defined in a neighborhood of $a \in \Re$. If $f,g$ are both locally strictly increasing (resp., decreasing) at $a$, then so are $\min \{ f,g \}$ and $\max \{ f,g\}$.
\end{rem}

Finally, we remark that in the proof of Theorem~\ref{thm:approx}, we use ideas also from \cite{DLS2,DHL,dumitrescu}.

\section{Proofs of Theorem~\ref{thm:rotational} and Corollary~\ref{cor:girth}}\label{sec:proof1}

First, we show how Theorem~\ref{thm:approx} implies Theorem~\ref{thm:rotational}.

Let $n \geq 3$ be a positive integer and let $\varepsilon > 0$ be a sufficiently small fixed value. By Theorem~\ref{thm:approx}, it is sufficient to construct a smooth convex body $K \in (1,m)_c$ for some value of $m$ with $n$-fold rotational symmetry and satisfying $d_H(K,\B^3) \leq \varepsilon$.
Let $P$ be a regular $n$-gon inscribed in a fixed circle $C$ on $\B^3$ parallel to, but not contained in the $(x,y)$-plane. Let the vertices of $P$ be $p_i$, $i=1,2,\ldots,n$. Let $Q (\varepsilon) = \conv \left( \B^3 \cup \{ (1+\varepsilon)p_1, \ldots, (1+\varepsilon)p_n \} \right)$ (cf. Figure~\ref{fig:thm1}).

\begin{figure}[ht]
\begin{center}
\includegraphics[width=0.3\textwidth]{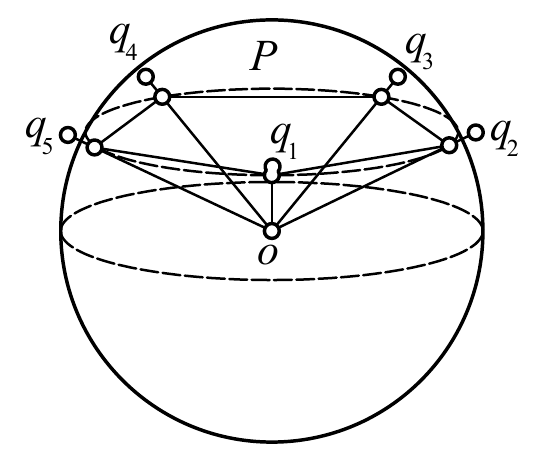}
\caption{The construction of $Q(\varepsilon)$ with $\varepsilon = 0.15$ and $n=5$, where $q_i = (1+\varepsilon) p_i$.}
\label{fig:thm1}
\end{center}
\end{figure}

Then $Q(\varepsilon)$ is the union of $\B^3$ and $n$ cones $C_i$, $i=1,2,\ldots,n$, with spherical circles centered at the points $p_i$ as directrixes.
By symmetry, the center of mass $c$ of $Q(\varepsilon)$ is on the $z$-axis, and by the Thales Theorem and Lemma~\ref{lem:approximation}, its distance from $o$ is of magnitude $O(\varepsilon^2)$. Thus, the points $(1+\varepsilon) p_i$ are equilibrium points of $Q(\varepsilon)$ if $\varepsilon$ is sufficiently small.
Furthermore, we have $c \neq o$. On one hand, from this we have that there are exactly two equilibrium points of $Q(\varepsilon)$ on $\Sph^2$, namely the points $(0,0,1)$ and $(0,0,-1)$, and exactly one of these points is stable, and the other one is unstable. On the other hand, this also implies that $Q(\varepsilon)$ has exactly one equilibrium point on each cone $C_i$ apart from its vertex; this point is a saddle point in the relative interior of a generating segment of $C_i$ (cf. Figure~\ref{fig:thm_rot}).

\begin{figure}[ht]
\begin{center}
\includegraphics[width=0.4\textwidth]{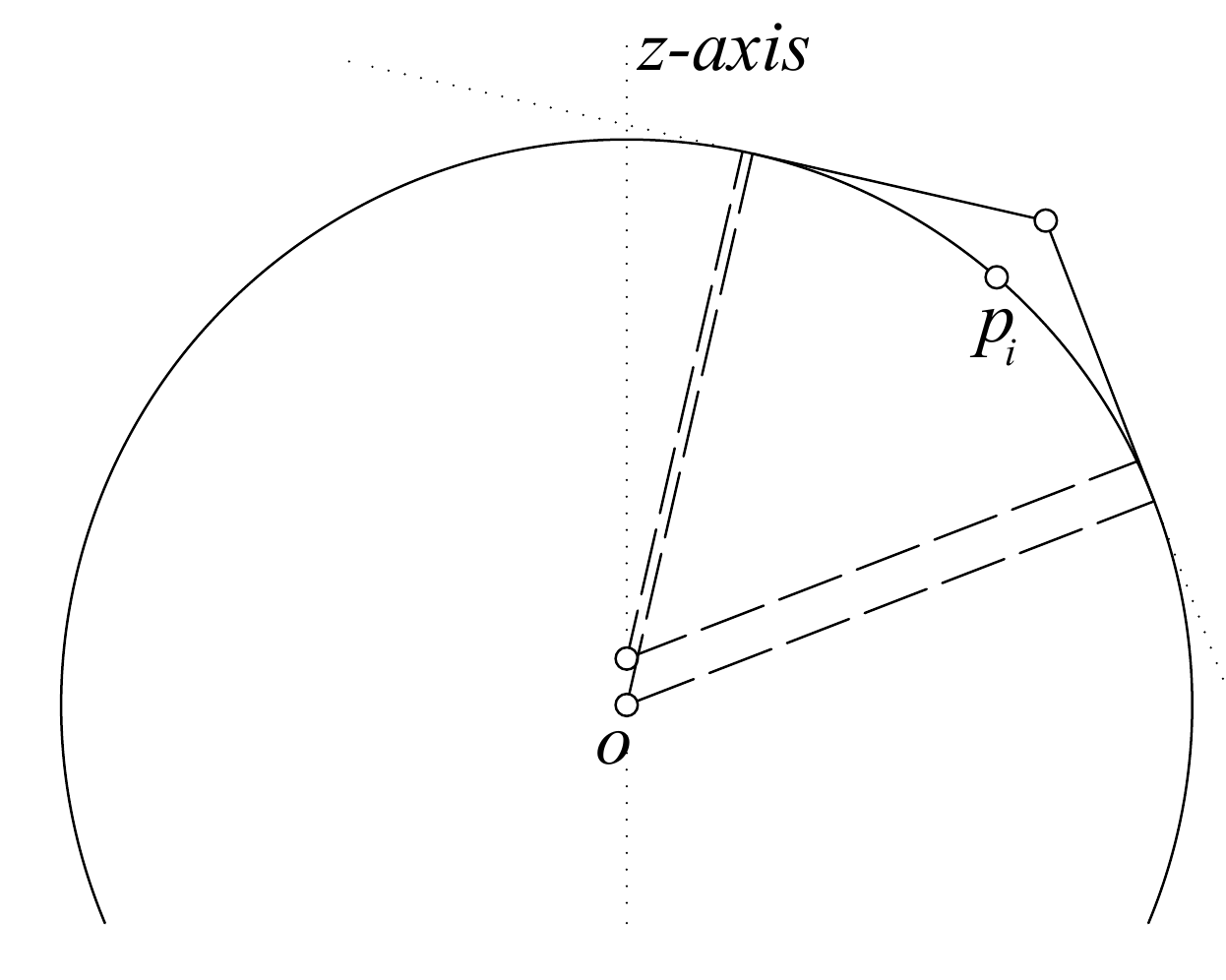}
\caption{An illustration for the proof of Theorem~\ref{thm:rotational}: equilibrium points on a conic part of $\bd (Q(\varepsilon)$.}
\label{fig:thm_rot}
\end{center}
\end{figure}

Now, we set $Q'(\varepsilon) = (Q(\varepsilon) \div (\tau \B^3)) +(\tau \B^3)$, where $\tau > 0$ is negligible compared to $\varepsilon$. Then $Q'(\varepsilon)$ is a smooth convex body which has $1$ stable, $n$ saddle-type and $(n+1)$ unstable points by Lemma~\ref{lem:approximation}. To guarantee that the body has positive principal curvatures at each equilibrium point, we may replace the generating segments of the cones by circular arcs of radius $R > 0$, where $\frac{1}{R}$ is negligible compared to $\tau$. The obtained convex body $K(\varepsilon) \in (1,n+1)_c$ satisfies the required conditions.

Finally, we prove Corollary~\ref{cor:girth}. Clearly, $\diam(\B^3)=2$ and $g(\B^3) = 2\pi$, and hence, the first statement follows from the continuity of diameter and girth with respect to Hausdorff distance.
On the other hand, it is well known that for any convex body $K$ in $\Re^3$, $\frac{\diam(K)}{g(K)} \geq \frac{1}{\pi}$; here we include the proof only for completeness. Let $K \subset \Re^3$ be a convex body, and let $w(\cdot)$ and $\perim(\cdot)$ denote mean width and perimeter, respectively. Then, for any projection $M$ of $K$, we have $w(M) \leq \diam(M) \leq \diam(K)$. On the other hand, it is well known that $w(M)= \frac{\perim(M)}{\pi}$, implying that $\frac{\diam(K)}{\perim(M)} \geq \frac{1}{\pi}$. From this, we readily obtain $\frac{\diam(K)}{g(K)} \geq \frac{1}{\pi}$.

\section{Proof of Theorem~\ref{thm:approx}}\label{sec:proof2}
                      
First, observe that by Lemma~\ref{lem:discrete}, $G$ is finite.

We construct $P$ by truncating $K$ with finitely many suitably chosen planes; or more precisely by taking its intersection with finitely many suitably chosen closed half spaces. We carry out the construction of $P$ in three steps.

In Step 1, we replace some small regions of $\bd (K)$ by polyhedral regions disjoint from all equilibrium points of $K$. These polyhedral regions will serve as `controlling regions'; that is, after constructing a polyhedron with $S$ stable and $U$ unstable points relative to $o$, we modify these regions to move back the center of mass of the constructed polyhedron to $o$.
In Step 2, we truncate a neighborhood of each equilibrium point to replace it by a polyhedral surface in such a way that each polyhedral surface contains exactly one equilibrium point relative to $o$, and the type of this point is the same as the type of the corresponding equilibrium point of $K$. 
Finally, in Step 3 we truncate the remaining part of $\bd (K)$ such that no new equilibrium point is created.
We describe these steps in three separate subsections.

In the proof, we denote by $\E$ the set of the equilibrium points of $K$, and for any point $q \in \bd(K)$, we denote by $H_q$ the unique supporting plane of $K$ at $q$. Observe that by the definition of $(S,U)_c$, $H_q \cap K = \{ q \}$ for any $q \in \E$, and set $X = \bd(K) \setminus \mathcal{E}$. Finally, by $F$ we denote the set of the fixed points of $G$, and note that $F$ is a linear subspace of $\Re^3$ that contains the center of mass of any $G$-invariant convex body.

\subsection{Step 1: Truncating some small regions disjoint from all equilibrium points}

We distinguish two cases depending on $\dim (F)$.\\

\textbf{Case 1}, if $F = \Re^3$.\\
By Carath\'eodory's theorem, there are points $z_1,z_2,z_3,z_4 \in X$ such that $o \in \conv \{z_1,z_2,z_3,z_4 \}$. Since $X$ is open in $\bd K$, we may choose these points to satisfy $o \in \inter \conv \{z_1,z_2,z_3,z_4 \}$. By the definition of $(S,U)_c$, we have that $H_{z_i}$ is disjoint from $H_q$ for any $1 \leq i \leq 4$ and $q \in \E$. We show that the $z_i$s can be chosen such that the planes $H_{z_i}$ are pairwise distinct.
Suppose for contradiction that, say, three of these planes coincide. Without loss of generality, assume that $H_{z_1}=H_{z_2}=H_{z_3}$, and denote this common plane by $H$. Then there are points $z_1', z_2', z_3' \in \mathrm{relbd} (K \cap H)$ such that $\conv \{ z_1, z_2, z_3\} \subseteq \conv \{ z_1', z_2', z_3'\}$. Now we may replace $z_2'$ and $z_3'$ by two points $z_2'',z_3'' \notin H$ such that $z_i''$ is sufficiently close to $z_i'$ for $i=2,3$. Then we have $o \in \inter \conv \{ z_1',z_2'', z_3'', z_4\}$, where no supporting plane of $K$ contains three of the points.
If a supporting plane of $K$ contains two of these points, we may repeat the above procedure, and finally obtain some points $w_1, \ldots, w_4 \in X$ such that $o \in \inter \conv \{ w_1, \ldots, w_4\}$, and the sets $H_{w_i} \cap K$ are pairwise disjoint.

Let $\delta > 0$, and let us truncate $K$ by planes $H_1, \ldots, H_4$ such that for all $i$s $H_i$ is parallel to $H_{z_i}$ and it is at the distance $\delta$ from it in the direction of $o$. We denote the truncated convex body by $K'$ and its center of mass by $c'$. By Remark~\ref{rem:stability}, if $\delta$ is sufficiently small, then $K$ has $S$ stable and $U$ unstable equilibrium points relative to $c'$, and each such equilibrium point is contained in $\bd (K') \setminus (\bigcup_{i=1}^4 H_i)$. Furthermore, if $\delta$ is sufficiently small, then for any point $q \in H_i \cap \bd (K) $ and any plane $H$ supporting $K'$ at $q$, $q$ is not perpendicular to $H$. Finally, since $o \in \inter \conv \{ w_1, \ldots, w_4 \}$, we may choose points $w_i' \in \relint (H_i \cap K')$ such that $c' \inter \conv \{ w_1', \ldots, w_4' \}$. For any $w_i'$, choose some convex $n_i$-gon $P_i \subset \relint (H_i \cap K')$ such that the center of mass of $P_i$ is $w_i'$ (cf. Figure~\ref{fig:step1}). Now we obtain the body $K''$ by truncating $K'$ by $n_i$ planes almost parallel to $H_i$ such that for each $i$, every side of $P_i$ is contained in one of the truncating planes, and we have $P_i = H_i \cap K''$. We choose the truncating planes such that the center of mass $c''$ of $K''$ satisfies $c'' \in \inter \conv \{ w'_1, \ldots, w'_4 \}$, and $K''$ has $S$ stable and $U$ unstable points on the smooth part of its boundary, and no equilibrium point on the non-smooth part. Now, we set $K_1=K''-c''$, $w_i''=w_i'-c''$ and $P_i'=P_i-c''$ for all $i$s, and for some sufficiently small $\bar{\tau} > 0$ we define four $1$-parameter families of polyhedral cones $C_i(\tau_i)= \conv (P_i' \cup \{ (1+\tau_i) w_i'' \})$, $\tau_i \in [0,\bar{\tau}]$, $i=1,2,3,4$. Furthermore, for later use, we set $K_0 = K - c''$, and call the set $X_1 = K_1 \cap \bd(K_0)$
the \emph{non-truncated part of $\bd(K_1)$}.

\begin{figure}[ht]
\begin{center}
\includegraphics[width=.7\textwidth]{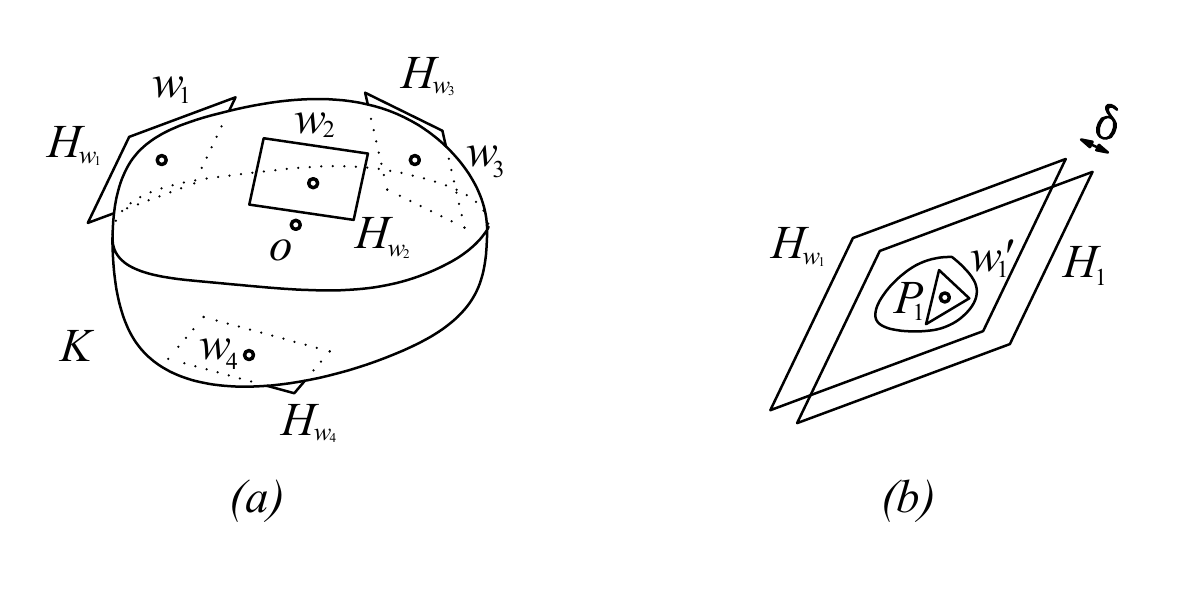}
\caption{An illustration for Step 1. Panel (a): The points $w_i$ and the planes supporting $K$ at these points. Panel (b): the polygon $P_1$ constructed on the intersection $H_1 \cap K'$.}
\label{fig:step1}
\end{center}
\end{figure}

If $\bar{\tau}$ is sufficiently small, then $K_1 \cup \bigcup_{i=1}^4 C_i(\tau_i)$ is convex for all values of the parameters $\tau_i$. Furthermore, note that the center of mass of $C_i(\tau_i)$ lies on the line through $[o,w''_i]$ for any value of $i$. This yields that the first moment of $\bigcup_{i=1}^4 C_i(\tau_i)$ is $\sum_{i=1}^4 (\alpha_i \tau_i + \beta_i \tau_i^2 ) w_i''$ for some suitable constants $\alpha_i, \beta_i >0$, which implies that it is surjective in a neighborhood of $o$. Thus, since $K_1$ is centered, after we replace the non-truncated part of $\bd (K_1)$ by a polyhedral surface in Steps 2 and 3, we may choose values of the $\tau_i$s in such a way that the sum of the first moment of $\bigcup_{i=1}^4 C_i(\tau_i)$ and of the first moment of the polyhedron $P$ obtained after Step 3 is equal $o$. This makes the polyhedron $P \cup \bigcup_{i=1}^4 C_i(\tau_i)$ centered.
Finally, we observe that by choosing sufficiently small values of $\delta$ and $\bar{\tau}$, for all values of the parameters, no point of $C_i(\tau_i)$ is an equilibrium point of $K_1$ relative to $o$.\\

\textbf{Case 2}, if $F \neq \Re^3$.\\
In this case $F$ is a plane or a line through $o$, or $F= \{ o \}$.
Consider the case that $F$ is a plane. Then, by the properties of isometries, the orbit of any point $p$ under $G$ consists of $p$ and its reflection about $F$. 
Let $K_F = F \cap K$, and observe that since $K$ is symmetric about $F$, for any $q \in \bd (K)$ $H_q$ is either disjoint from $K_F$ or $q \in  \mathrm{relbd} (K_F)$. 
Thus, we may apply the argument in Case 1 for $K_F$, and obtain some points $z_1, z_2, z_3 \in X \cap K_F$ such that $o \in \relint \conv \{ z_1, z_2, z_3 \}$ and the planes $H_{z_i}$ are pairwise disjoint. But then there are some points $z_3'$ and $z_3''$, sufficiently close to $H_{z_3}$ such that $z_3''$ is the reflected copy of $z_3'$ about $F$, $o \in \inter \conv \{ z_1, z_2, z_3', z_3'' \}$, and the supporting planes at these points are pairwise disjoint. Clearly, the set $\{ z_1, z_2, z_3', z_3'' \}$ is $G$-invariant. From now on, we may apply the argument in Case 1.

If $F$ is a line, we may apply a similar argument. Indeed, let $F \cap K = [p,p']$. Then, by symmetry, $p$ and $p'$ are equilibrium points with respect to $o$, and thus, the planes $H$ and $H'$, containing $p$ and $p'$, and perpendicular to $[o,p]$ and $[o,p']$, respectively, support $K$. Thus, we may choose a point $w_1$, with $\{ w_1, w_2, \ldots, s_k \}$ as its orbit, such that $w_1$ is sufficiently close to $H$, and the orbit of $H_1 \cap K$, where $H_1$ is the supporting plane of $K$, contains mutually disjoint sets. Choosing points $w_1', w_2', \ldots, w_k'$ similarly and sufficiently close to $H'$, we have $o \in \inter \conv \{ w_1, \ldots, w_k, w_1', \ldots, w_k' \}$, and we may proceed as in Case 1.

Finally, if $F = \{ o \}$, then any $G$-invariant convex body (and in particular the convex polyhedron constructed in Steps 2 and 3) is centered. Thus, in this case we may skip Step 1. 

Based on the existence of the families $C_i(\tau_i)$, in Steps 2 and 3 all equilibrium points are meant to be \emph{relative to $o$}. We denote by $\E_1$ the set of the equilibrium points of $K_1$.

\subsection{Step 2: Truncating small neighborhoods of equilibrium points}

In this step we take all points $q \in \E_1$, and truncate neighborhoods of them in $\bd (K_1)$ simultaneously for all points in the orbit of $q$. Here we observe that the orbit of an equilibrium point consists of equilibrium points. We carry out the truncations in such a way that the regions truncated in Step 1 or Step 2 are pairwise disjoint. We denote the convex body obtained in this step by $K_2$, and set $X_2 = \bd (K_1) \cap K_2$.
We construct $K_2$ in such a way that for any point $p \in X_2$ there is no supporting plane $H$ of $K_2$ through $p$ which contains an equilibrium point of $K_2$.

Consider some $q \in \E_1$. Without loss of generality, we may assume that $q=(0,0,\rho)$ for some $\rho > 0$, and denote by $e_x, e_y,$ and $e_z$ the vectors of the standard orthonormal basis. With a little abuse of notation, for any $p \in \bd(K_0)$, we denote by $H_p$ the unique supporting plane of $K_0$ at $p$.\\

\textbf{Case 1}, the stabilizer of $q$ in $G$ is the identity; i.e. $q$ not fixed under any element of $G$ other than the identity.

\emph{Subcase 1.1}, $q$ is a stable point of $K_1$.\\
In this case we truncate $K_1$ by a plane $H_q'$ parallel to, and sufficiently close to $H_q$. Then we truncate $K_1$ by finitely many additional planes such that any point of $H_q' \cap \bd(K_1)$ is truncated by at least one of them, and for any point $p$ of the non-truncated part $X_2$ of $\bd(K_1)$ there is no supporting plane $H$ of $K_2$ through $p$ which contains an equilibrium point of $K_2$ relative to $o$.

\emph{Subcase 1.2}, $q$ is a saddle-type equilibrium point.\\
Note that by Remark~\ref{rem:curvature}, $q$ is not an umbilic point of $\bd (K_1)$, and its principal curvatures $\kappa_1 < \kappa_2$ satisfy the inequalities 
$0 < \kappa_1 < \frac{1}{\rho} < \kappa_2$.

Without loss of generality, we may assume that the sectional curvature of $\bd(K_1)$ in the $(x,z)$-plane is $\kappa_1$, and in the $(y,z)$-plane it is $\kappa_2$. For any $\tau > 0$, let $K_1(\tau)$ denote the set of points of $K_1$ with $z$-coordinates at least $\rho - \tau$, and observe that by the fact that $\kappa_2 > \kappa_1 > 0$, for any $\varepsilon > 0$ there is some $\tau > 0$ such that $K_1(\tau)$ is contained in the neighborhood of $q$ of radius $\varepsilon$. For any $\{i,j \} \subset \{x,y,z\}$, let $H_{ij}$ denote the $(i,j)$ coordinate plane, and $\proj_{ij}$ denote the orthogonal projection of $\Re^3$ onto $H_{ij}$.

For any $\eta> 0$, let $C(\eta)$ be the set of the points of the circular disk $y^2+(z-\rho+\eta)^2 \leq {\eta}^2$ in $H_{yz}$ whose $z$-coordinates are at least $\rho - \tau$. Then, since $\bd(K_1)$ is $C^2$-class in a neighborhood of $q$, we have that for any $\eta_1, \eta_2 > 0$ satisfying $\frac{1}{\rho} < \frac{1}{\eta_1} < \kappa_2 < \frac{1}{\eta_2}$, if $\tau$ is sufficiently small, then $C(\eta_2) \subseteq \proj_{yz}( K_1(\tau) ) \subseteq C(\eta_1)$ holds. Since $\proj_{yz}(K_1)$ is convex, $\mathrm{relbd}(\proj_{yz}(K_1))$ has exactly two points with their $z$-coordinates equal to $\rho - \tau$. Let these points be $q^-=(0,\sigma^-,\rho-\tau)$ and $q^+=(0,\sigma^+,\rho-\tau)$ such that $\sigma^- < 0 < \sigma^+$ (cf. Figure~\ref{fig:step2}). Then there are some supporting lines $L_-, L_+$ of $\proj_{yz}( K_1)$ passing through $q^-$ and $q^+$, respectively.

\begin{figure}[ht]
\begin{center}
\includegraphics[width=0.5\textwidth]{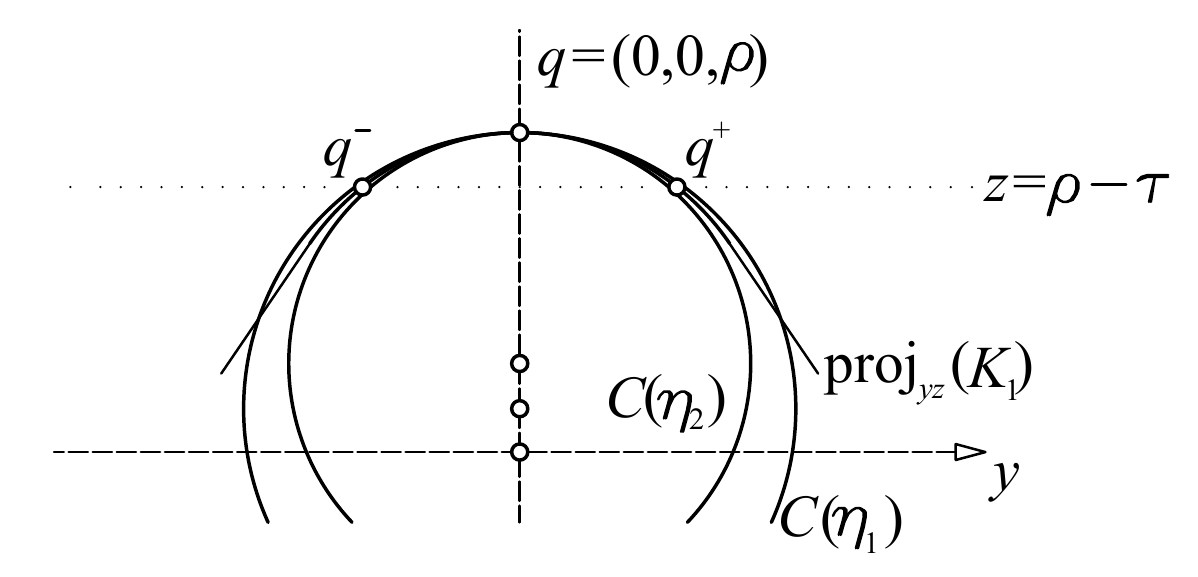}
\caption{An illustration for Step 2 in the proof of Theorem~\ref{thm:approx}.}
\label{fig:step2}
\end{center}
\end{figure}

Clearly, for $i \in \{ -, +\}$, $L^i$ is the orthogonal projection of some supporting plane $H^i$ of $K_1$ onto $H_{yz}$.
Let $r^i$ be a point of $L^i$, on the open half line starting at $q^i$ such that the line through $[o,q^i]$ do not separate $q$ and $r_i$. Then, for any fixed values of $\eta_1$ and $\eta_2$ and sufficiently small value of $\tau$, the angles $o q^i r^i \angle$ are obtuse.
Now we choose some sufficiently small value of $\zeta > 0$, and define ${q^i}'= \zeta e_z + q^i$, ${r^i}'=\zeta e_z + r^i$, ${L^i}' = \zeta e_z + L^i$, ${H^i}' = \zeta e_z + H^i$ and $q'=-\zeta e_z + q$. Then we may assume that $|{q^i}'| < |q'|$, the angles $o {q^i}' {r^i}' \angle$ are obtuse, and the planes ${H^i}'$ are disjoint from $K_1$.

Thus, by Lemma~\ref{lem:conway_lemma}, for $i \in \{ -,+ \}$, there is a polygonal curve $\Gamma_i$ in $H_{yz}$, connecting $q'$ to ${q^i}'$ such that
the Euclidean distance measured from the points of $\Gamma_i$ to $o$ is strictly decreasing as we move from $q'$ to ${q^i}'$ (see the remark after Lemma~\ref{lem:conway_lemma}), $\Gamma_i$ is contained in $\mathrm{relbd} (\conv (\Gamma_i \cup \{ o \}))$, and the latter set is supported by ${L_i}'$ in $H_{yz}$.
Consider the closed, convex set $C_H \subset H_{yz}$ bounded by $\Gamma_- \cup \Gamma_+$, the half line of ${L^+}'$ starting at ${q^+}'$ and not containing ${r^+}'$, and the half line of ${L^-}'$ starting at ${q^-}'$ and not containing ${r^-}'$, and set $C= \proj_{yz}^{-1} (C_H) \subset \Re^3$. By the previous consideration, $C$ is an infinite convex cylinder with the properties that $o \in \inter(C)$, $K_1 \setminus C \subseteq K_1(\tau)$, and the equilibrium points of $C$ relative to $o$ are $q$ and two stable points on ${L_+}'$ and ${L_-}'$, respectively. To construct $K_2$, we truncate $K_1$ by $C$, and show that, apart from the saddle point $q'$, no new equilibrium point is created by this truncation. 

Observe that by our construction, any new equilibrium point is a point of $\bd (K_1) \cap \bd(C)$. Suppose that there is some equilibrium point $q \in \bd(K_1) \cap \bd(C)$ of $K_1 \cap C$. To reach a contradiction, we identify $H_{xy}$ with $\Re^2$, and parametrize $\bd(K_1)$ in a neighborhood of $q$ as the graph of a function $f : \Re^2 \to \Re$ and $\bd(C)$ in a neighborhood of $q'$ as the graph of a function $g:\Re^2 \to \Re$. Note that by the nondegeneracy of $q$, for some value of $\phi > 0$, $o$ has a neighborhood $U \subset \Re^2$ such that for any $w=(x_0,y_0) \in U$ whose angle with the $x$-axis is at most $\phi$, $|(x,y,f(x,y)|$ is locally strictly increasing at $w$ as a function of $x$ if $x_0 > 0$ and locally strictly decreasing if $x_0 < 0$; furthermore, if the angle of $w$ with the $y$-axis is at most $\phi$, then $|(x,y,f(x,y)|$ is locally strictly decreasing as a function of $y$ if $y_0 > 0$ and locally strictly increasing if $y_0 < 0$. Note that by Remark~\ref{rem:monotone}, the same property holds for $\min \{ |(x,y,f(x,y))|, |(x,y,g(x,y))| \}$ as well. Observe that since $q$ is an equilibrium point of $K_1 \cap C$, it is an equilibrium point of the section of $K_1 \cap C$ with the plane through $q$ parallel to $H_{xy}$. Thus, by Remark~\ref{rem:forsaddle}, if $\tau > 0$ is chosen small enough, then the angle of $\proj_{xy}(q)$ with the $x$-axis or the $y$-axis is at most $\phi$. But this contradicts our previous observation that at such a point 
$\min \{ |(x,y,f(x,y))|, |(x,y,g(x,y))| \}$ is locally strictly increasing or decreasing parallel to the $x$- or the $y$-axis.

Finally, to exclude the possibility that a support plane of $K_1 \cap C$ through a point in $\bd(K_1) \cap \bd(C_1)$ contains $q'$, we truncate all points of $\bd(K_1) \cap \bd(C)$ by planes, not containing $q'$, whose intersections with $K_1 \cap C$ do not contain equilibrium point.

\emph{Subcase 1.3}, $q$ is an unstable point.\\
In this case both principal curvatures $\kappa_1, \kappa_2$ of $\bd(K_1)$ at $q$ satisfy $\kappa_1, \kappa_2 >  \frac{1}{\rho} > 0$, and thus, there is some constant $\max \left\{ \frac{1}{\kappa_1}, \frac{1}{\kappa_2} \right\} < \eta < \rho$ such that the ball $\frac{\rho-\eta}{\rho} q + \eta \B^3$ contains a neighborhood of $q$ in $\bd(K_1)$. We parametrize $\bd(K_1)$ in a neighborhood of $q$ as the graph $\{ z=f(x,y) \}$ of a function $(x,y) \mapsto f(x,y)$, and note that by nondegeneracy, there is some $r_0 > 0$ such that the function $|(r \cos (\phi), r \sin (\phi) ,f( r \cos (\phi) , r \sin (\phi)))|$ is a strictly decreasing function of $r$ on $[0,r_0]$ for all values of $\phi$.

For any $\tau > 0$, let $K_1(\tau)$ denote the set of points of $K_1$ with $z$-coordinates at least $\rho - \tau$, let $H_{\tau}$ denote the plane with equation $\{ z = \rho-\tau \}$. Let $\tau > 0$ be sufficiently small. Then there is a circle $C_0$ centered at $(0,0,\rho-\tau)$ which is contained in $H_{\tau} \cap \inter (\eta \B^3)$, and is disjoint from $K_1$. Let $H$ be a plane supporting $K_1$ at a point of $H_{\tau}$ such that its angle $\alpha$ with $H_{\tau}$ is minimal among these supporting planes. Let $H'$ be the translate of $H$ touching $C_0$ such that $H$ strictly separates $o$ and $H'$, and
let the intersection point of $H'$ and the $z$-axis be $r$. Consider the infinite cone $C$ with apex $r$ and base $C_0$, and observe that it contains $K_1 \setminus K_1(\tau)$ in its interior. Now, let $q'= q - \zeta e_z$ for some suitably small $\zeta > 0$, and let $\Gamma$ be a polygonal curve connecting $q'$ to a point $p \in C_0$ such that the plane of $o,p,q'$ contains $\Gamma$, $\Gamma \subset \mathrm{relbd} (\conv (\Gamma \cup \{ o \}))$, and the Euclidean distance function is strictly decreasing along $\Gamma$ from $q'$ to $p$. Let $L_p$ denote the closed half line in the line of $[r,p]$ starting at $p$ and not containing $r$, and let $\Gamma' = \Gamma \cup L_p$. Let $m \geq 3$ be arbitrary, and for any $i=0,1,\ldots, m-1$, let $\Gamma'_i$ denote the rotated copy of $\Gamma'$ around the $z$-axis, with angle $\frac{2\pi i}{m}$. Let $P' = \conv \bigcup_{i=0}^{m-1} \Gamma'_i$. Then $P'$ is a convex polyhedral domain such that $K_1 \setminus P' \subseteq K_1(\tau)$, and if $m$ is sufficiently large, then at any boundary point of $P'$ with $z$-coordinate greater than $\rho - \tau$, $|(x,y,g(x,y))|$ is strictly locally increasing in a neighborhood of $(0,0)$ as a function of $\sqrt{x^2+y^2}$, where $\bd(P')$ is given as the graph of the function $z=g(x,y)$. Thus, by Remark~\ref{rem:monotone} and following the idea at the end of Subcase 1.2 in Step 2, we may truncate a neighborhood of $q$ in $\bd(K_1)$ by a convex polyhedral region $P'$ such that the only equilibrium point of the truncated body on $\bd(P')$ is the unstable point $q'$, and the truncated body has no non-truncated boundary point where some supporting plane contains an equilibrium point. 

The procedure discussed in Subcases 1.1-1.3 for $q$ is applied for any equilibrium point in the orbit of $q$ in an analogous way.\\

\textbf{Case 2}, if the stabilizer of $q$ in $G$ is not the identity.\\
In this case the procedure described in Case 1 is carried out in such a way that the truncating polyhedral domain is invariant under any element of $G$ fixing $q$.

Summing up, to construct $K_2$ in Step 2 we truncated a neighborhood of each equilibrium point of $K_1$ by a polyhedral region in such a way that each region contains exactly one equilibrium point relative to $o$, and no plane supporting $K_2$ at any point of $X_2=\bd(K_1) \cap K_2$ contains an equilibrium point of $K_2$ relative to $o$. In addition, $K_2$ is $G$-invariant.

\subsection{Step 3: Truncating the remaining part of the boundary}

In this step let $Y = X_1 \cap X_2$.
Furthermore, for any plane $H$ let $o_H$ denote the orthogonal projection of $o$ onto $H$, and
let $\mathcal{H}$ denote the family of planes $H$ with the property that $o_H \in K_2$. Note that $\mathcal{H}$ consists of
\begin{itemize}
\item all planes through $o$, and, 
\item for any $p \in K_2 \setminus \{ o \}$, the (unique) plane passing through $p$ and perpendicular to $[o,p]$.
\end{itemize}
Observe that $Y$ is compact, and by our construction, for any plane $H$ supporting $K_2$ at some point $p \in Y$, we have $o_H \notin H \cap K_2$; or equivalently, we have $\{ o_H : H \in \mathcal{H} \} \cap Y = \emptyset$. Thus, by compactness, there is some $\delta > 0$ such that for any $H \in \mathcal{H}$ and $p \in Y$, the distance of $o_H$ and $p$ is at least $\delta > 0$.
Now, for any $p \in Y$, let $H_p$ denote the unique closed half space whose boundary supports $K_0$ at $p$ and which satisfies $\inter(H_p) \cap K_1 = \emptyset$. Let $u_p$ denote the outer unit normal vector of $K$ at $p$, and for any $\zeta > 0$, set $H_p(\zeta) = H_p - \zeta u_p$, and $Y(\zeta) = K_2 \cap \left( \bigcup_{p \in Y} H_p(\zeta) \right)$ (cf. Figure~\ref{fig:step3}). Clearly, $Y(\zeta)$ tends to $Y$ with respect to Hausdorff distance as $\zeta \to 0^+$.
Thus, there is some sufficiently small $\zeta_0 > 0$ such that for any $H \in \mathcal{H}$, $o_H \notin Y(\zeta_0)$.

\begin{figure}[ht]
\begin{center}
\includegraphics[width=0.4\textwidth]{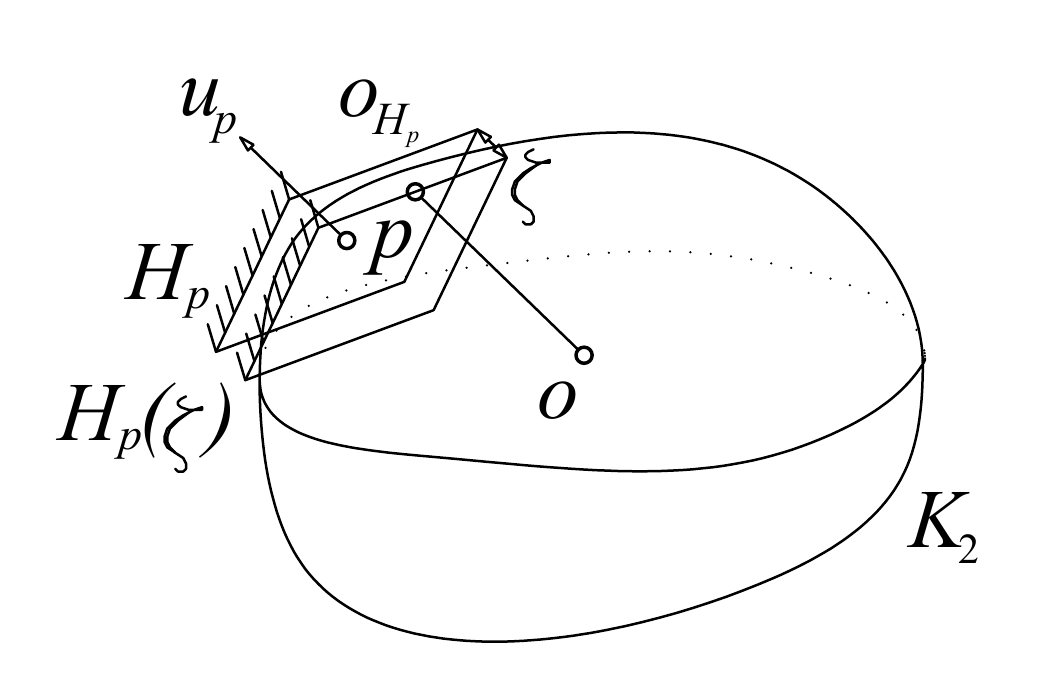}
\caption{An illustration for Step 3 in the proof of Theorem~\ref{thm:approx}.}
\label{fig:step3}
\end{center}
\end{figure}

Now, for any $p \in Y$, set $U(p) = Y \cap \inter (H_p(\zeta_0))$. Then $U(p)$ is an open neighborhood of $p$ in $Y$. Thus, by the compactness of $Y$, there are finitely many points $p_1, \ldots, p_m$ such that $\bigcup_{i=1}^m U(p_i) = Y$. Then, clearly $P=K_2 \cap \left(\bigcap_{i=1}^m \left( \Re^3 \setminus \inter (H_{p_i}(\zeta_0)) \right) \right)$ is a convex polytope contained in $K_2$. Furthermore, since $\zeta_0 > 0$ can be arbitrarily small, $P$ can be arbitrarily close to $K_2$.

We show that no point $q \in \bd(P)$, contained in some $\bd (H_{p_i}(\zeta_0))$ is an equilibrium point of $P$. Indeed, if $q$ was such a point, then the plane $H$ through $q$ and perpendicular to $[o,q]$ is contained in $\mathcal{H}$. On the other hand, $q \in \bd(P) \cap \bd (H_{p_i}(\zeta_0)) \subset Y(\zeta_0)$, which is impossible by our choice of $\zeta_0$.

Finally, we may choose the points $p_1, p_2, \ldots, p_m$ in such a way that the set $\{ p_1, \ldots, p_m \}$ is invariant under the act of any element of $G$.

\section{Remarks and open questions}\label{sec:remarks}

First, we remark that by using truncations instead of conic extensions in the proof of Theorem~\ref{thm:rotational}, we readily obtain Theorem~\ref{thm:rot2}. Here, a \emph{mono-unstable} convex body is meant to be a nondegenerate convex body with a unique unstable point.

\begin{thm}\label{thm:rot2}
For any $n \geq 3$, $n \in \mathbb{Z}$ and $\varepsilon > 0$ there is a homogeneous mono-unstable polyhedron $P$ such that $P$ has an $n$-fold rotational symmetry and
$d_H(P, \B^3) < \varepsilon$.
\end{thm}

We ask the following.

\begin{ques}
What are the positive integers $n \geq 2$ such that class $(1,1)_p$ contains a convex polyhedron with an $n$-fold axis of symmetry?
\end{ques}

In light of the words of Shephard in \cite{Shephard} from Section~\ref{sec:intro} about monostable polyhedra, we remark that a consequence of Theorem~\ref{thm:approx} is that to study the metric properties of nondegenerate polyhedra, in particular monostable polyhedra, it is sufficient to study the metric properties of their smooth counterparts, which seem to be much more tractable.

Next, we conjecture that Theorem~\ref{thm:approx} remains true if we omit the condition that the principal curvatures of $\bd(K)$ at every equilibrium point of $K$ are strictly positive.

Finally, to propose a conjecture for the first part of Problem~\ref{prob3}, we recall the following concept from \cite{gomboc2}, where the function $\rho_K : \Sph^2 \to \Re$, $\rho_K (x) = \max \{ \lambda : \lambda x \in K \}$ is called the \emph{gauge function} of the convex body $K$.

\begin{defn}
Let $K \in \Re^3$ be a centered convex body, and for any simple, closed curve $\Gamma \subset \Sph^2$, let $\Gamma^+$ and $\Gamma^-$ denote the two compact, connected domains in $\Sph^2$ bounded by $\Gamma$.
Then the quantities
\[
F(K) = \sup_{\Gamma} \sup_{p_1 \in \Gamma^+, p_2 \in \Gamma^-} \frac{\min \{ \rho_K(s) : s \in \Gamma \} }{\max \{ \rho_K(p_1), \rho_K(p_2) \} }
\]
and
\[
T(K) = \sup_{\Gamma} \sup_{p_1 \in \Gamma^+, p_2 \in \Gamma^-} \frac{\min \{ \rho_K(p_1), \rho_K(p_2) \} }{\max  \{ \rho_K(s) : s \in \Gamma \} }
\]
are called the \emph{flatness} and the \emph{thinness} of $K$, respectively.
\end{defn}

Domokos and V\'arkonyi in \cite{gomboc2} proved that for any nondegenerate, centered smooth convex body $K$, $F(K)=1$ if and only if $K$ is monostable, and $T(K)=1$ if and only if $K$ is mono-unstable.

Recall that a nondegenerate convex body is \emph{mono-monostatic} if it has a unique stable and a unique unstable point \cite{gomboc}.
We conjecture the following.

\begin{conj}\label{conj:connected}
For any centered convex body $K \subset \Re^3$, $K$ can be uniformly approximated by monostable convex polyhedra if and only if $F(K)=1$, by mono-unstable convex polyhedra if and only if $T(K)=1$, and by mono-monostatic convex polyhedra if and only if $F(K)=T(K)=1$.
\end{conj}

\bigskip

\emph{Acknowledgements}. 
The author expresses his gratitude to G\'abor Domokos and, more generally, the members of the Morphodynamics Research Group with whom he had many fruitful discussions on the subject, and also to an unknown referee for many helpful comments.

\end{document}